\documentclass[11pt]{article}
\usepackage{amsfonts,mathtools, mathrsfs}
\usepackage{amsmath,amsthm, bbm}
\usepackage[usenames,dvipsnames]{color}
\usepackage[toc]{appendix} 
\usepackage{hyperref}

\newcommand{\beq}{\begin{equation}}
\newcommand{\eeq}{\end{equation}}

\newcommand{\bea}{\begin{eqnarray}}
\newcommand{\eea}{\end{eqnarray}}
\newcommand{\beas}{\begin{eqnarray*}}
\newcommand{\eeas}{\end{eqnarray*}}

\newtheorem{theorem}{Theorem}[section]

\newtheorem{definition}[theorem]{Definition}
\newtheorem{proposition}[theorem]{Proposition}

\newtheorem{corollary}[theorem]{Corollary}
\newtheorem{lemma}[theorem]{Lemma}
\newtheorem{remark}[theorem]{Remark}
\newtheorem{example}[theorem]{Example}
\newtheorem{examples}[theorem]{Examples}
\newtheorem{foo}[theorem]{Remarks}

\newcommand{\ep}{\epsilon}

\newcommand{\E}{\mathbb E}
\newcommand{\C}{\mathbb C}

\newcommand{\bpartial}{\overline{\partial}}
\newcommand{\bw}{\overline{w}}

\newcommand{\tr}{\mathrm tr}

\title{Generalized stochastic areas, Winding numbers, and hyperbolic Stiefel fibrations}

\author{Fabrice Baudoin\footnote{Author supported in part by NSF grant  DMS-1901315}, Nizar Demni, Jing Wang\footnote{Author supported in part by NSF grant  DMS-1855523}}

\usepackage{todonotes}
\mathtoolsset{showonlyrefs=true}
\begin{document}
\maketitle

\begin{abstract}
We study the Brownian motion on the non-compact Grassmann manifold $\frac{\mathbf{U}(n-k,k)} {\mathbf{U}(n-k)\mathbf{U}(k)}$ and some of its functionals. The key point is to realize this Brownian motion as a matrix diffusion process, use matrix stochastic calculus and take advantage of the hyperbolic Stiefel fibration to study a functional that can be understood in that setting as a generalized stochastic area process. In particular, a connection to the generalized Maass Laplacian of the complex hyperbolic space is presented and applications to the study of Brownian windings in the Lie group $\mathbf{U}(n-k,k)$ are then given.
\end{abstract}

\tableofcontents

\section{Introduction}\label{Sec1}

In this introduction, we  first explain how the results of this paper fit into a  larger research project concerning the study of integrable functionals of Brownian motions on symmetric spaces. We then present our main results and how  the paper is structured.

\subsection*{Context}

Let $G$ be a Riemannian Lie group and $K,H$ two compact subgroups of $G$ with $K \subset H$. According to a standard construction due to L. B\'erard-Bergery (see \cite[Theorem 9.80]{Besse}), under some natural compatibility properties of the metric, one has a Riemannian fibration
\[
 H/K \to G/K \to G/H.
\]
One can construct interesting functionals of the Brownian motion on $G/H$ by lifting this Brownian motion to $G/K$ and then look at the fiber motion in $H/K$  of that lift using a skew-product type decomposition. When the coset space $G/H$ is a  symmetric space, in many cases it has been shown to produce functionals which remarkably turn out to be integrable in the sense that their Laplace transforms can be expressed using special functions from harmonic analysis. So far, this construction and the probabilistic study of those functionals and of their distributions, was carried out in details in the following cases:

\begin{itemize}
\item $G=\mathbf{U}(n)$, $K=\mathbf{U}(n-1)$, $H=\mathbf{U}(n-1)\mathbf{U}(1)$. In that case $G/H$ is the complex projective space $\mathbb{C}P^{n-1}$, and the B\'erard-Bergery fibration reduces to the classical  Hopf fibration $\mathbf{U}(1) \to \mathbb{S}^{2n-1} \to \mathbb{C}P^{n-1}$. This case was studied   in \cite{BW-SWHF}.
\item $G=\mathbf{U}(n-1,1)$, $K=\mathbf{U}(n-1)$, $H=\mathbf{U}(n-1)\mathbf{U}(1)$. In that case $G/H$ is the complex hyperbolic space $\mathbb{C}H^{n-1}$, and the B\'erard-Bergery fibration is the anti-de Sitter fibration  $\mathbf{U}(1) \to \mathbf{AdS}^{2n-1} \to \mathbb{C}H^{n-1}$. This case was studied   in \cite{BW-SWHF}.
\item $G=\mathbf{Sp}(n)$, $K=\mathbf{Sp}(n-1)$, $H=\mathbf{Sp}(n-1)\mathbf{Sp}(1)$. In that case $G/H$ is the quaternionic projective space $\mathbb{H}P^{n-1}$, and the B\'erard-Bergery fibration is the quaternionic  Hopf fibration $\mathbf{Sp}(1) \to \mathbb{S}^{4n-1} \to \mathbb{H}P^{n-1}$. This case was studied in \cite{BDW-area}.
\item $G=\mathbf{Sp}(n-1,1)$, $K=\mathbf{Sp}(n-1)$, $H=\mathbf{Sp}(n-1)\mathbf{Sp}(1)$. In that case $G/H$ is the quaternionic hyperbolic space $\mathbb{H}H^{n-1}$, and the B\'erard-Bergery fibration is the quaternionic  anti de-Sitter fibration $\mathbf{Sp}(1) \to \mathbf{AdS}_\mathbb{H}^{4n-1} \to \mathbb{H}H^{n-1}$. This case was studied in \cite{BDW-area}.

\item $G=\mathbf{Spin}(9)$, $K=\mathbf{Spin}(7)$, $H=\mathbf{Spin}(8)$. In that case $G/H$ is the octonionic projective line $\mathbb{O}P^{1} \simeq \mathbb{S}^8$, and the B\'erard-Bergery fibration is the octonionic  Hopf fibration $\mathbb{S}^7 \to \mathbb{S}^{15} \to \mathbb{S}^8$. This case was partially studied   in \cite{Baudoin-Cho}, see also \cite{CY}.

\item $G=\mathbf{U}(n)$, $K=\mathbf{U}(n-k)$, $H=\mathbf{U}(n-k)\mathbf{U}(k)$ with $k \ge 1$. In that case $G/H$ is the complex Grassmannian $G_{n,k}$ and the B\'erard-Bergery fibration, the Stiefel fibration, This case was studied   in \cite{BW-Winding}.
\end{itemize}

In the present paper, we complete further this list and focus on the case $G=\mathbf{U}(n-k,k)$, $K=\mathbf{U}(n-k)$, $H=\mathbf{U}(n-k)\mathbf{U}(k)$ with $k \ge 1$. The symmetric space $G/H =\mathbf{U}(n-k,k) / \mathbf{U}(n-k)\mathbf{U}(k)$  is then the dual symmetric space of $\mathbf{U}(n) / \mathbf{U}(n-k)\mathbf{U}(k)$, so in a sense our results in this paper are the \textit{hyperbolic} counterparts of the \textit{spherical} results in \cite{BW-Winding}.

\subsection*{Main results}

Our goal in this paper is to study the Brownian motion $(w_t)_{t \ge 0}$ and some of its functionals in the hyperbolic  complex Grassmann manifold $HG_{n,k}= \mathbf{U}(n-k,k) / \mathbf{U}(n-k)\mathbf{U}(k)$. 

As a first result, we show how to realise $(w_t)_{t \ge 0}$ as a matrix diffusion process. More precisely,  let $U_t=\begin{pmatrix}  Y_t & X_t \\ W_t & Z_t \end{pmatrix}$ be a Brownian motion on the Lie group of matrices $\mathbf{U}(n-k,k)$.  
Then, in Theorem \ref{main:s1}, we show that the process $(w_t)_{t \ge 0}:=\left( X_t  Z_t^{-1}\right)_{t \ge 0} $ is a matrix diffusion process in $\mathbb{C}^{(n-k)\times k}$ with generator $\frac12\Delta_{HG_{n,k}}$, where $\Delta_{HG_{n,k}}$ is the Laplace-Beltrami operator of $HG_{n,k}$ computed in a global system of coordinates. We shall refer to the latter as inhomogeneous by analogy with the usual inhomogeneous coordinates on the complex hyperbolic space $\mathbb{C}H^{n-1}$ which corresponds to $k=1$ (see e. g. \cite{BW-SWHF}). Realizing the Brownian motion on $HG_{n,k}$ as a matrix diffusion process has the advantage to make available all the tools from stochastic calculus and random matrix theory.  In particular, one can write an explicit stochastic differential equation (hereafter SDE) for the Markov process $J:=w^*w$  and after applying results from \cite{Gra-Mal}, deduce the one satisfied by the eigenvalue process of $J$. In this respect, we prove that almost surely, there is no collision between eigenvalues and relate the latter to an instance of the so-called Heckman-Opdam process. In particular, we readily deduce from \cite{Sch} that the eigenvalue process is the unique strong solution of the SDE it satisfies. 

In Section 3 we study in details the additive functional $\int_0^t \mathop{\tr} (J_s)ds$. Our interest in that functional comes from its close relationship to the generalized stochastic area functional that we later study. In this respect, our two main results are Proposition \ref{Lemma1}, Theorem \ref{theo-Laplace} and Proposition \ref{NewLaplace} below. More precisely, Proposition \ref{Lemma1} proves that almost surely, the following limiting behavior holds: 
\begin{equation*}
\lim_{t\to+\infty}\frac1t\int_0^t\tr(J_s)ds=k.
\end{equation*}
In particular, this limit theorem (and its proof) is very different from the corresponding one obtained in \cite{BW-Winding} (see Theorem 4.1. there). Indeed, It reflects the non-compactness of the Grassmann manifold $HG_{n,k}$ in the sense that the eigenvalues of $J$ all converge almost surely and more rapidly to the boundary at $\infty$. As to Theorem \ref{theo-Laplace}, it provides an explicit formula for the Laplace transform of $\int_0^t \mathop{\tr} (J_s)ds$ and its proof makes use of Girsanov Theorem together with the Karlin-McGregor formula. In Proposition \ref{NewLaplace}, we improve Theorem \ref{theo-Laplace} through the connection between the hyperbolic Jacobi operator and the radial part of the generalized Maass Laplacian in the complex hyperbolic space. Actually, we transform the Laplace formula obtained in Theorem \ref{theo-Laplace} into another one which looks more transparent regarding the above limiting result. For instance, we shall prove in the rank-one case $k=1$ how to recover the latter result from the new expression of the Laplace transform. Doing so gives the full expansion of the Laplace transform in the time variable and for small spectral parameters. However, since the computations are already tricky and tedious even in this particular case, we postpone them in an appendix at the end of the paper. 

In Section 4, we appeal to the hyperbolic Stiefel fibration
\[
\mathbf{U}(k) \to HV_{n,k} \to HG_{n,k},
\]
where $HV_{n,k}=\mathbf{U}(n-k,k)/\mathbf{U}(n-k)$, to provide a skew-product decomposition of the  Brownian motion on $HV_{n,k}$. The horizontal $HG_{n,k}$-part of that Brownian motion is nothing else but a Brownian motion $w$ on $HG_{n,k}$.
As to the fiber $\mathbf{U}(k)$-part, it is the stochastic development of the following Stratonovich line integral: 
\begin{equation*}
\int_{w[0,t]} \eta,
\end{equation*}
where $\eta$ is a $\mathfrak{u}(k)$-valued one-form on $HG_{n,k}$ such that $d \mathrm{tr} (\eta)$ is the K\"ahler form of the Riemannian K\"ahler space $HG_{n,k}$. Therefore, in the sense of  \cite{BW-SWHF}, the fiber functional may be interpreted as the generalized stochastic area process of the Brownian motion $w$. Moreover, it turns out that 
\begin{align*}
\int_{w[0,t]} \mathrm{tr} (\eta) 
 &=i\mathcal{B}_{\int_0^t \mathrm{tr}(J)ds}
\end{align*}
where $\mathcal{B}$ is a one-dimensional Brownian motion independent of $J$. Therefore the distribution of $\int_{w[0,t]} \mathrm{tr} (\eta)$ is directly related to that  of $\int_0^t\tr(J_s)ds$ already studied in details in Section 3. In particular, one obtains
from Proposition \ref{Lemma1} together with a time scaling the following limit in distribution as $t\to+\infty$:
\[
\frac{1}{i\sqrt{t}} \int_{w[0,t]} \mathrm{tr} (\eta) \to \mathcal{N}(0,k).
\]
Surprisingly, this limiting distribution has finite moments of all orders in contrast to the compact case for which the Cauchy distribution is obtained (Proposition 4.5. in \cite{BW-Winding}).  

Finally, as an application of our methods and results, we prove the following asymptotic winding theorem.

\begin{theorem}
Let $U_t=\begin{pmatrix}  Y_t & X_t \\ W_t & Z_t \end{pmatrix}$ be a Brownian motion on the Lie group $\mathbf{U}(n-k,k)$ with $1 \le k \le n-k$.  One has the polar decomposition
\[
\det (Z_t)=\varrho_t e^{i\theta_t}
\]
where $\varrho_t \ge 1$ is a continuous semimartingale  and $\theta_t$ is a continuous martingale such that the following convergence holds in distribution when $t \to +\infty$
\[
\frac{\theta_t}{\sqrt{t}} \to \mathcal{N}(0,2k) .
\]

\end{theorem}
This is the hyperbolic analogue of Theorem 1.1. in \cite{BW-Winding} and reminds of the long time behavior of windings of cusp forms around geodesics in modular surfaces (\cite{Lejan-Gui}, Theorem 2.1.). 

\subsection*{Notations}

Throughout the paper we will use the following notations.

\begin{itemize}

\item If $M \in \mathbb{C}^{n \times n}$ is a $n \times n$ matrix with complex entries, we will denote $M^*=\overline{M}^T$ its adjoint.

\item If $z_j=x_j +i y_j$ is a complex coordinate system, then
\[
\frac{\partial }{\partial z_j }= \frac{1}{2} \left(  \frac{\partial }{\partial x_j } -i \frac{\partial }{\partial y_j }\right) , \quad \frac{\partial }{\partial \bar{z}_j }= \frac{1}{2} \left(  \frac{\partial }{\partial x_j } +i \frac{\partial }{\partial y_j }\right)
\]
are the Wirtinger operators. 
\item Throughout the paper we work on a filtered probability space $(\Omega, (\mathcal{F}_t)_{t \ge 0} , P)$ that satisfies the usual conditions.

\item If $X$ and $Y$ are semimartingales, we denote by $\int X dY$ the It\^o integral,  by $\int X \circ dY$ the Stratonovich integral and by $\int dX dY$ or $\langle X, Y \rangle$ the quadratic variation. Of course, all these integrals are to be understood entry-wise. 
For instance, the quadratic variation $\int dMdN$ is the matrix whose entries are 
\begin{equation*}
\left(\int dXdY\right)_{ij}= \sum_{\ell} \int dX_{i\ell}dY_{\ell j}.
\end{equation*}

\item If $M$ is a semimartingale and $\eta$ a one-form, then $\int_{M[0,t]} \eta$ denotes the Stratonovich line integral of $\eta$ along the paths of $M$.

\end{itemize}

\section{Brownian motion on  hyperbolic complex Grassmannian manifolds}\label{Sec2}

\subsection{The hyperbolic complex Grassmannian manifold and inhomogeneous coordinates}\label{sec:notations}

Let $n\in \mathbb{N}$, $n \ge 2$, and $k\in\{1,\dots, n-1\}$. We consider the non-compact Lie group of matrices which belongs to the family of classical Lie groups of matrices (see \cite{Hec-Sch}, \cite{Hel}):
\[
\mathbf{U}(n-k,k)=\left\{ M \in GL(n,\mathbb{C}), M^* \begin{pmatrix} I_{n-k}   & 0 \\ 0 & -I_{k} \end{pmatrix} M =  \begin{pmatrix} I_{n-k}   & 0 \\ 0 & -I_{k} \end{pmatrix}\right\}.
\]
The Lie algebra of $\mathbf{U}(n-k,k)$ is given by
\[
\mathfrak{u}(n-k,k)=\left\{ A \in \C^{n \times n}, A^* \begin{pmatrix} I_{n-k}   & 0 \\ 0 & -I_{k} \end{pmatrix} + \begin{pmatrix} I_{n-k}   & 0 \\ 0 & -I_{k} \end{pmatrix}A= 0 \right\}.
\]
Note that the real dimension of $\mathfrak{u}(n-k,k)$ and hence $\mathbf{U}(n-k,k)$ is $n^2$. 

We consider the complex hyperbolic Stiefel manifold $HV_{n,k}$ defined as the set
\begin{equation}\label{eq-HV}
HV_{n,k}=\left\{ \begin{pmatrix} X  \\ Z \end{pmatrix}  \in \C^{n\times k} , X \in  \C^{(n-k)\times k}, Z \in \C^{k\times k} | X^* X-Z^*Z=-I_k \right\}.
\end{equation}
It is a non-compact algebraic real sub-manifold of $\C^{n\times k}$ with real dimension $2nk-k^2$.
We note that $\mathbf{U}(n-k,k)$ acts transitively on $HV_{n,k}$, the action being  defined by $(g,M) \to gM$, $g \in \mathbf{U}(n-k,k)$, $M \in HV_{n,k}$. For this action, the isotropy group of $\begin{pmatrix} 0 \\ I_k \end{pmatrix} \in HV_{n,k}$ is simply the unitary group $\mathbf{U}(n-k)$, where $\mathbf{U}(n-k)$ is identified with the normal subgroup:
\[
\left\{   \begin{pmatrix}  Y  & 0 \\ 0 & I_{k} \end{pmatrix} , Y \in \mathbf{U}(n-k) \right\}.
\]
Therefore $HV_{n,k}$ can be identified with the coset space $\mathbf{U}(n-k,k) / \mathbf{U}(n-k)$. 

In addition, there is a right  action of the  unitary group $\mathbf{U}(k)$ on $HV_{n,k}$, which is given by the right matrix multiplication: $M  g$, $M \in HV_{n,k}$,  $g \in \mathbf{U}(k)$. The quotient space by this action 
$HG_{n,k}:=HV_{n,k} / \mathbf{U}(k)$ is called the \textit{complex hyperbolic Grassmannian manifold} and we shall denote by $\pi$ the projection $HV_{n,k}\to HG_{n,k}$. 
Besides, $HG_{n,k}$ is a non-compact complex manifold with complex dimension $k(n-k)$. Actually, as a  manifold, $HG_{n,k}$ is diffeomorphic to $\C^{(n-k)\times k}$. Indeed, consider the smooth map $p : HV_{n,k} \to \C^{(n-k)\times k}$ given by $p \begin{pmatrix} X  \\ Z \end{pmatrix}  = X Z^{-1}.$ It is clear that for every $g \in \mathbf{U}(k)$ and $M \in HV_{n,k}$, $p( M  g)= p(M)$. Since $p$ is a submersion from $HV_{n,k}$ onto its image $ p ( HV_{n,k}) = \C^{(n-k)\times k} $ we easily deduce that there exists a unique  diffeomorphism $\Psi$ from   $HG_{n,k}$ onto $\C^{(n-k)\times k}$ such that 
\begin{align}\label{inh-coord}
 \Psi \circ \pi= p.
 \end{align}
 The map $\Psi$ induces a global coordinate chart on $HG_{n,k}$ that we call inhomogeneous and through those global coordinates, we will often simply identify $HG_{n,k} \simeq \C^{(n-k)\times k}$ .
 
 For the purpose of constructing and studying the Brownian motion on $HG_{n,k}$ we need to equip $HG_{n,k}$ with a Riemannian metric. To proceed, set 
  \[
 I_{n-k,k} :=\begin{pmatrix} I_{n-k}   & 0 \\ 0 & -I_{k} \end{pmatrix}.
 \]
and note that $A \in \mathfrak{u}(n-k,k)$ if and only if $I_{n-k,k}A$ is skew-Hermitian. Therefore, 
 $$\langle A, B\rangle_{\mathfrak{u}(n-k,k)}=-\frac12 \tr(I_{n-k,k}AI_{n-k,k}B),$$
is a non-degenerate $Ad$-invariant inner product on $\mathfrak{u}(n-k,k)$ where $Ad$ is the adjoint action of $U(n-k,k)$ on $\mathfrak{u}(n-k,k)$.  Moreover, an orthonormal basis of $\mathfrak{u}(n-k,k)$ with respect to this inner product is given by
\begin{align*}
 &\{ T_\ell, 1 \le \ell \le n \} \\
&\cup \{E_{\ell j}-E_{j \ell}, i(E_{\ell j}+E_{j\ell}),  1\le \ell <j \le n-k, n-k+1\le \ell <j \le n\}\\
&\cup\{E_{\ell j}+E_{j \ell}, i(E_{\ell j}-E_{j\ell}), 1\le \ell\le n-k, n-k+1\le j\le n\}
\end{align*}
where $E_{j\ell}=(\delta_{j\ell}(k,h))_{1\le k,h\le n}$ and $T_j=\sqrt{2}i E_{jj}$. In this way, we obtain a bi-invariant Riemannian structure on $\mathbf{U}(n-k, k)$. 

The inner product above also equips the coset space $HV_{n,k} \simeq \mathbf{U}(n-k,k) / \mathbf{U}(n-k)$ with the unique Riemannian structure such that  the map $ M \in \mathbf{U}(n-k, k) \to M \begin{pmatrix} 0 \\ I_k \end{pmatrix} \in HV_{n,k} $ is a Riemannian submersion. Finally, we can equip the hyperbolic Grassmannian $HG_{n,k} $  with the Riemannian metric such that the map 
\begin{equation*}
p: \begin{pmatrix} X  \\ Z \end{pmatrix}  \in HV_{n,k} \mapsto X Z^{-1} \in HG_{n,k}
\end{equation*}
 is a Riemannian submersion with totally geodesic fibers isometric to $\mathbf{U}(k)$. Equipped with that metric, $HG_{n,k} $ is a complex Riemannian symmetric of non compact type and rank $\min (k,n-k) $. It is the dual symmetric space of the complex Grassmannian space $G_{n,k} $ that was considered in \cite{BW-Winding}.

We therefore have a fibration
\[
\mathbf{U}(k) \to HV_{n,k} \to HG_{n,k}
\]
which we will  referred to as the hyperbolic Stiefel fibration. Note that this fibration is a special case of Theorem 9.80 in \cite{Besse} with $G=\mathbf{U}(n-k,k)$, $H=\mathbf{U}(n-k)\mathbf{U}(k)$, $K=\mathbf{U}(n-k)$.

For $k=1$, $HV_{n,1}$ is isometric to the complex anti-de Sitter space  $\mathbf{AdS}^{2n-1}$ equipped with its Riemannian metric, $HG_{n,1}$ is the complex hyperbolic space $\mathbb{C} H^{n-1}$ and the above fibration is the anti-de Sitter fibration considered in \cite{BW-SWHF}.  

On the other hand, due to the symmetric space representation $HG_{n,k} \simeq \mathbf{U}(n-k,k)/ \mathbf{U}(n-k)\mathbf{U}(k)$, one has the duality $HG_{n,k} \simeq HG_{n,n-k}$ which lead to the fibration
\[
\mathbf{U}(n-k) \to HV_{n,n-k} \to HG_{n,n-k}.
\]
Therefore without loss of generality and unless specified otherwise, we will always assume throughout the paper that $k \le n-k$. As such, the rank of $HG_{n,k}$ as a symmetric space is $k$.

 \subsection{Brownian motion on $HG_{n,k}$}
 
 In this section, we study the Brownian motion on $HG_{n,k}$ and show how it can be constructed from a Brownian motion on the indefinite Lie group $\mathbf{U}(n-k,k)$.
 
  A Brownian motion $(A_t)_{t\ge0}$ on the Lie algebra $\mathfrak{u}(n-k,k)$ is  given by
\begin{align}\label{eq-A_t}
A_t &=\sum_{1\le \ell<j\le n-k}(E_{\ell j}-E_{j\ell})B^{\ell j}_t+i(E_{\ell j}+E_{j\ell})\tilde{B}^{\ell j}_t+\sum_{j=1}^k T_{j}\hat{B}^j_t\notag\\
&+\sum_{n-k+1\le \ell<j\le n}(E_{\ell j}-E_{j\ell})B^{\ell j}_t+i(E_{\ell j}+E_{j\ell})\tilde{B}^{\ell j}_t+\sum_{j=n-k+1}^n T_{j}\hat{B}^j_t\\
&+\sum_{1\le \ell \le n-k}\sum_{n-k+1\le j\le n}(E_{\ell j}+E_{j\ell})B^{\ell j}_t+i(E_{\ell j}-E_{j\ell})\tilde{B}^{\ell j}_t\notag
\end{align}
where $B^{\ell j}$, $\tilde{B}^{\ell j}$, $\hat{B}^j$ are  independent standard real Brownian motions.
In the following, we will  use the block notations as below: For any $U\in \mathbf{U}(n-k,k)$ 
\[
U=\begin{pmatrix}  Y & X \\ W & Z \end{pmatrix},\quad
\]
where $X\in \C^{(n-k)\times k}$, $Y\in \C^{(n-k)\times (n-k)}$, $Z\in \C^{k\times k}$, $W\in \C^{k\times (n-k)}$. We note then  that  $\begin{pmatrix} X  \\ Z \end{pmatrix}  \in HV_{n,k}$.
The block notation for $A\in\mathfrak{u}(n-k,k)$ is 
\[
A=\begin{pmatrix} \ep & \gamma \\  \beta  &  \alpha \end{pmatrix},\quad
\]
where $\ep\in \C^{(n-k)\times (n-k)}$, $\beta\in \C^{k\times (n-k)}$, $\gamma\in \C^{(n-k)\times k}$, $\alpha\in \C^{k\times k}$. Clearly $\ep$ and $\alpha$ are skew-Hermitian and 
\[
\beta^*=\gamma.
\]
Let $(U_t)_{t\ge0}$ be a $\mathbf{U}(n-k,k)$-valued stochastic process that satisfies the Stratonovich differential equation
\begin{equation}\label{eq-bm-sde}
\begin{cases}
dU_t=U_t\circ dA_t ,\\
U_0=\begin{pmatrix}  Y_0 & X_0 \\ W_0 & Z_0 \end{pmatrix}.
\end{cases}
\end{equation}
\begin{definition}
The process $(U_t)_{t \ge 0}$ is a Brownian motion  on $\mathbf{U}(n-k,k)$ starting from $U_0$.
\end{definition}
%
%

 The main theorem of the section is the following:

\begin{theorem}\label{main:s1}
Let $U_t=\begin{pmatrix}  Y_t & X_t \\ W_t & Z_t \end{pmatrix}$ be a Brownian motion on the Lie group $\mathbf{U}(n-k,k)$.  Then, the process $(w_t)_{t \ge 0}:=\left( X_t  Z_t^{-1}\right)_{t \ge 0} $ is a diffusion process  with  generator given by the diffusion operator $\frac12\Delta_{HG_{n,k}}$, where 
\begin{align*}
\Delta_{HG_{n,k}}&={{4}}\sum_{1\le i, i'\le n-k, 1\le j, j'\le k}(I_{n-k}-ww^*)_{ii'} (I_k-w^*w)_{j'j}\frac{\partial^2}{\partial w_{ij}{\partial \overline{w}_{i'j'}}}.
\end{align*}
Moreover, $(w_t)_{t \ge 0}$ is a Brownian motion in $HG_{n,k}$, i.e. $\Delta_{HG_{n,k}}$ is the Laplace-Beltrami operator on $HG_{n,k}$ expressed in inhomogeneous coordinates.
\end{theorem}
\begin{proof}
First, we note that both the maps 
\begin{equation*}
p_1: \begin{pmatrix}  Y & X \\ W & Z \end{pmatrix} \in \mathbf{U}(n-k,k) \to \begin{pmatrix} X  \\ Z \end{pmatrix}  \in HV_{n,k}
\end{equation*}
and $p: HV_{n,k} \rightarrow \mathbb{C}^{(n-k) \times k} \sim HG_{n,k}$ are Riemannian submersions with totally geodesic fibers. Thus, they transform Brownian motions into Brownian motions so that $w_t=(p \circ p_1)(U_t)$ is indeed a Brownian motion on $HG_{n,k}$. It remains to compute explicitly its generator.

From \eqref{eq-bm-sde} we have
\begin{align}\label{eq-H-matrix-sde}
dX &=X \circ d\alpha +Y \circ d\gamma =X  d\alpha +Y d\gamma +\frac12 (dX  d\alpha +dY  d\gamma )\nonumber \\
dY &=X \circ d\beta +Y \circ d\ep =X  d\beta +Y  d\ep +\frac12 (dX  d\beta +dY  d\ep)\nonumber\\
dZ &= Z \circ d\alpha +W \circ d\gamma =Z  d\alpha +W d\gamma +\frac12 (dZ  d\alpha +dW  d\gamma )\\
dW &=Z \circ d\beta +W \circ d\ep =Z  d\beta +W  d\ep +\frac12 (dZ  d\beta +dW  d\ep ).\nonumber
\end{align}
Moreover form \eqref{eq-A_t} we have that
\[
d\alpha d\alpha=-2kI_kdt,\quad d\ep d\ep=-2(n-k)I_{n-k}dt
\]
and 
\[
d\beta d\gamma= d\beta d\beta^* = 2(n-k)I_kdt,\quad d\gamma d\beta=2k I_{n-k}dt.
\]
We then have
\begin{align*}
dX &=X \circ d\alpha +Y \circ d\gamma =X  d\alpha +Y  d\gamma +(n-2k) X dt \\
dZ &= Z \circ d\alpha +W \circ d\gamma =Z  d\alpha +W  d\gamma +(n-2k) Z dt.
\end{align*}

Now the process $w=XZ^{-1}$ satisfies:
\[
dw =dX  Z^{-1}+X dZ^{-1}+dXdZ^{-1}.
\]
But $Z dZ^{-1}=-dZ  Z^{-1}-dZdZ^{-1}$ whence 
we have
\begin{align}\label{eq-dw}
&dw = dX  Z ^{-1}-w dZ  Z^{-1}-w dZdZ^{-1}+dX dZ^{-1}\nonumber \\
&=(X d\alpha+Y d\gamma+(n-2k)X dt) Z^{-1}-w (Z d\alpha+W d\gamma+(n-2k)Zdt) Z^{-1}-w dZdZ^{-1}+dXdZ^{-1}\nonumber  \\
&=Y d\gamma Z^{-1}-w W d\gamma Z^{-1}-w dZdZ^{-1}+dX dZ^{-1}.
\end{align}
Note that for the finite variation part of $dw$ we have
\begin{align*}
&-w dZdZ^{-1}+dX dZ ^{-1}=w dZ Z^{-1}dZ Z^{-1}-dX\, Z^{-1}dZ\, Z^{-1}\\
&\qquad=w(Zd\alpha+Wd\gamma)Z^{-1}(Zd\alpha+Wd\gamma)Z^{-1}-(X d\alpha+Y d\gamma)Z^{-1}\, (Z d\alpha+W d\gamma)Z^{-1}\\
&\qquad=w Z d\alpha d\alpha Z^{-1}-X(d\alpha d\alpha)Z^{-1}=0.
\end{align*}
Hence we have
\[
dw=Y d\gamma Z^{-1}-w W d\gamma Z^{-1}.
\]
We can now prove that $w$ is a matrix diffusion process using the above formula. Since for $1\le i\le n-k$, $1\le j\le k$,
\[
dw_{ij}=\sum_{\ell=1}^{n-k}(Y-w W )_{i\ell} (d\gamma Z^{-1})_{\ell j},
\]
we have
\[
dw_{ij}d\overline{w}_{i'j'}=\sum_{\ell,m=1}^{n-k}(Y-w W )_{i\ell}(\overline{Y}-\overline{w W} )_{i'm} (d\gamma Z^{-1})_{\ell j}(d\overline{\gamma} \overline{Z}^{-1})_{m j'}.
\]
But
\[
(d\gamma Z^{-1})_{\ell j}(d\overline{\gamma} \overline{Z}^{-1})_{m j'}
=\sum_{p,q=1}^k (d\gamma)_{\ell p}(Z^{-1})_{p j}(d\overline{\gamma})_{mq} (\overline{Z}^{-1})_{qj'}=2\delta_{m\ell} dt\sum_{p=1}^k(Z^{-1})_{p j} (\overline{Z}^{-1})_{pj'},
\]
whence
\begin{align}\label{eq-dwdw-mid}
dw_{ij} d\overline{w}_{i'j'}
&=2 ((Y-w W )\overline{(Y-w W)}^T )_{ii'}((ZZ^*)^{-1})_{j'j}dt.
\end{align}
Now, if $M \in U(n-k, k)$ if and only if:
$$ \begin{pmatrix} Y^*   & W^* \\ X^* & Z^* \end{pmatrix} \begin{pmatrix} I_{n-k}   & 0 \\ 0 & -I_{k} \end{pmatrix}\begin{pmatrix}  Y & X \\ W & Z \end{pmatrix}=
\begin{pmatrix}  Y & X \\ W & Z \end{pmatrix} \begin{pmatrix} I_{n-k}   & 0 \\ 0 & -I_{k} \end{pmatrix}\begin{pmatrix} Y^*   & W^* \\ X^* & Z^* \end{pmatrix}=\begin{pmatrix} I_{n-k}   & 0 \\ 0 & -I_{k} \end{pmatrix},$$ 
Equivalently,

\begin{equation}\label{eq-relations1}
X^*X-Z^*Z=-I_{k}, \quad X^*Y-Z^*W=0,\quad Y^*Y-W^*W=I_{n-k}
\end{equation}
and
\begin{equation}\label{eq-relations}
YY^*-XX^*=I_{n-k}, \quad ZX^*-WY^*=0,\quad WW^*-ZZ^*=-I_{k}.
\end{equation}

From \eqref{eq-relations}, we get the following relations
\[
wWY^* = XZ^{-1}WY^*=XX^*, \quad I_k-w^*w = I_k - (Z^{-1})^*X^*XZ^{-1} = (Z^{-1})^*Z^{-1}, 
\]
which give after substitution in \eqref{eq-dwdw-mid}:
\begin{align}\label{eq-dwdw}
dw_{ij}d\overline{w}_{i'j'}&=2(I_{n-k}-ww^*)_{ii'}((ZZ^*)^{-1})_{j'j}dt\nonumber \\
&=2(I_{n-k}-ww^*)_{ii'}(I_k-w^*w)_{j'j}dt.
\end{align}
Therefore, we conclude that $(w_t)_{t\ge0}$ is a diffusion whose generator is given by $\frac12\Delta_{HG_{n,k}}$.
\end{proof}

\begin{remark}
When $k=1$, the diffusion operator $\Delta_{HG_{n,1}}$ coincides with the Laplacian of $\mathbb{CH}^n$ in inhomogeneous coordinates whose formula was already known (see for instance  \cite{BW-SWHF}).
\[
\Delta_{\mathbb{C}H^{n-1}}=4(1-|w|^2)\sum_{k=1}^{n-1} \frac{\partial^2}{\partial w_k \partial\overline{w_k}}- 4(1-|w|^2)\mathcal{R} \overline{\mathcal{R}}
\]
where
\[
\mathcal{R} :=\sum_{j=1}^{n-1} w_j \frac{\partial}{\partial w_j}.
\] 
\end{remark}

\subsection{Invariant measure}
Since $HG_{n,k}$ is a non-compact symmetric space, its Riemannian volume measure has infinite mass. Our goal in this section will be to compute explicitly its density measure in inhomogeneous coordinates up to a scalar multiple. We will take advantage of the explicit formula for $\Delta_{HG_{n,k}}$ obtained in Theorem \ref{main:s1}. Consider in inhomogeneous coordinates the following measure
\[
d\mu:=\det(I_k-w^*w)^{-n} dm := \rho(w) dm
\]
where $m$ denotes the Lebesgue measure on $HG_{n,k  }\simeq \C^{(n-k)\times k}$ and we omit the dependence on $n$ for sake of simplicity.

In the proposition below we show that $\mu$ is a symmetric and invariant measure for $\Delta_{HG_{n,k}}$.
\begin{proposition}\label{invariant1}
The measure $\mu$ is invariant and symmetric for the operator $(1/2)\Delta_{HG_{n,k}}$. Namely for every smooth and compactly supported functions $f,g$ on $\C^{(n-k)\times k}$, we have: 
\[
\int (\Delta_{HG_{n,k}}f) g\, d\mu = \int f(\Delta_{HG_{n,k}}g)\, d\mu=-\int \Gamma_{\Delta_{HG_{n,k}}}(f, g)\, d\mu,
\]
where the carr\'e du champ operator 
\[
\Gamma_{\Delta_{HG_{n,k}}}(f,g):=\frac{1}{2} \left( \Delta_{HG_{n,k}}(fg) - (\Delta_{HG_{n,k}}f) g - (\Delta_{HG_{n,k}}g) f \right)
\]
 is explicitly given by
\begin{equation}\label{eq-Gamma}
\Gamma_{\Delta_{HG_{n,k}}}(f,g)=2\sum_{1\le i, i'\le n-k, 1\le j, j'\le k}(I_{n-k}-ww^*)_{ii'}(I_k-w^*w)_{j'j}\bigg(\frac{\partial f}{\partial w_{ij}}\frac{\partial g}{\partial \bw_{i'j'}}+\frac{\partial g}{\partial w_{ij}}\frac{\partial f}{\partial \bw_{i'j'}}\bigg).
\end{equation}
\end{proposition}
\begin{proof}
For ease of notations, we further set:  
\begin{equation*}
\partial_{ij} :=\frac{\partial }{\partial w_{ij}}, \quad  \bpartial_{i j} :=\frac{\partial }{\partial \overline{w}_{ij}}, \quad A_{ii'jj'} := (I_{n-k}-ww^*)_{ii'}(I_k-w^*w)_{j'j}.
\end{equation*}
so that 
\[
\Gamma_{\Delta_{HG_{n,k}}}(f,g) = 2\sum_{1\le i, i'\le n-k, 1\le j, j'\le k}A_{ii'jj'}\bigg((\partial_{ij}f)(\bpartial_{i'j'}g)+(\partial_{ij}g)(\bpartial_{i'j'}f)\bigg).
\]
By integration by parts, we have
\begin{align*}
&-\frac{1}{2}\int (\Delta_{HG_{n,k}}f) g\, d\mu \\
=& \sum_{1\le i, i'\le n-k, 1\le j, j'\le k} \int (\partial_{ij}f)\, \bpartial_{i'j'} (A_{ii'jj'}g\rho)dm+\int (\bpartial_{i'j'}f)\, \partial_{ij} (A_{ii'jj'}g\rho)dm\\
=&\frac{1}{2}\int \Gamma_{\Delta_{HG_{n,k}}}(f, g)\, d\mu + R 
\end{align*}
where 
\begin{multline*}
R := \sum_{1\le i, i'\le n-k, 1\le j, j'\le k} \bigg(\int [(\partial_{ij}f)\,(\bpartial_{i'j'} A_{ii'jj'})+(\bpartial_{i'j'}f)\,(\partial_{ij} A_{ii'jj'})]g\rho dm \\
+\int [(\partial_{ij}f)\,(\bpartial_{i'j'}\rho)+(\bpartial_{i'j'}f)\,(\partial_{ij}\rho)]gA_{ii'jj'} dm\bigg).
\end{multline*}
Hence, it remains to prove that $R=0$. To this end, we use the relations: 
\begin{align*}
\bpartial_{i'j'} A_{ii'jj'}=-w_{ij'}(\delta_{j'j}-(w^*w)_{j'j})-(\delta_{ii'}-(ww^*)_{ii'})w_{i'j},
\end{align*}
and
\[
\partial_{ij} A_{ii'jj'}=-\bw_{i'j}(\delta_{j'j}-(w^*w)_{j'j})-(\delta_{ii'}-(ww^*)_{ii'})\bw_{ij'}
\]
to get
\[
\sum_{1\le i'\le n-k, 1\le j'\le k} \bpartial_{i'j'} A_{ii'jj'}=-n({w}(I_k-J))_{ij}
\]
and
\[
\sum_{1\le i\le n-k, 1\le j\le k} \partial_{ij} A_{ii'jj'}=-n({\bw}(I_k-\bar{J}))_{i'j'}.
\]
Moreover, since
\begin{align*}
\bpartial_{i'j'}\det(I_k-J)&=\det(I_k-J)\sum_{1\le p,q\le k}\big((I_k-J)^{-1}\big)_{qp}\bpartial_{i'j'} (I_k-J)_{pq}\\
&=-\det(I_k-J)\bigg(w(I_k-J)^{-1} \bigg)_{i'j'}
\end{align*}
and
\[
\partial_{ij}\det(I_k-J)=-\det(I_k-J)\bigg(\bw(I_k-\bar{J})^{-1} \bigg)_{ij},
\]
then
\begin{align*}
\bpartial_{i'j'} \rho
&=n  \rho\bigg(w(I_k-J)^{-1}\bigg)_{i'j'}, \quad
\partial_{ij} \rho=n  \rho\bigg(\bw(I_k-\bar{J})^{-1} \bigg)_{ij}.
\end{align*}
As a result,
\begin{align*}
\sum_{1\le i'\le n-k, 1\le j'\le k}(\bpartial_{i'j'} A_{ii'jj'})g\rho + (\bpartial_{i'j'}\rho)gA_{ii'jj'}=0
\end{align*}
and
\[
\sum_{1\le i\le n-k, 1\le j\le k}(\partial_{ij} A_{ii'jj'})g\rho + (\partial_{ij}\rho)gA_{ii'jj'}=0,
\]
that is $R=0$ as claimed.
\end{proof}

\subsection{Eigenvalue process}\label{sec-J process}

Let $(w_t)_{t \ge 0}=(X_t Z_t^{-1})_{t \ge 0}$ be a Brownian motion on $HG_{n,k}$ as in Theorem \ref{main:s1}. Let  $J_t= w_t^*w_t \in \mathbb{C}^{k \times k}$ for $t \ge 0$.  We wish to study the eigenvalues process of $J$. The first goal is therefore to write a stochastic differential equation for $J$.

\begin{proposition}
Let $(J_t)_{t \ge 0}$ be given as above. Then, as long as $J_t,  t \geq 0,$ is invertible, there exists a Brownian motion $R$ in $\mathbb{C}^{ k \times k}$ such that:
\begin{equation}\label{SDE-Matrix}
dJ_t= \sqrt{I_k-J_t} dR\sqrt{I_{k} - J_t} \sqrt{J_t} + \sqrt{J_t} \sqrt{I_{k} - J_t}dR_t^{*}\sqrt{I_k-J_t} + 2((n-k) - \tr(J_t))(I_k - J_t) dt. 
\end{equation}
\end{proposition}
\begin{proof}
Since $J=w^*w$ we have
\begin{equation*}
dJ = (dw^*)w + w^*(dw) + (dw^*)(dw). 
\end{equation*}
From \eqref{eq-dwdw}, we readily derive: 
\begin{equation*}
(dw^*)(dw) = 2((n-k) - \tr(J))(I_k - J). 
\end{equation*}
As to the local martingale part of $dJ$, recall that that $I_{k} - w^*w = (Z^{-1})^*Z^{-1}$ is invertible and assume that $J$ is invertible up to time $t > 0$. Then $w$ has a unique polar decomposition $w = Q\sqrt{J}$ where $Q \in \mathbb{C}^{(n-k)\times k}$ satisfies: 
\begin{equation*}
QQ^* = I_{n-k}, \quad Q^*Q = I_k.
\end{equation*}
In particular, $I_{n-k} - ww^*$ is invertible as well since 
\begin{equation*}
I_{n-k} - ww^* = 0 \quad \Leftrightarrow \quad Q(I_k-J)Q^* = 0. 
\end{equation*}
As a matter fact, \eqref{eq-dwdw} implies the existence of a complex Brownian matrix $B \in \mathbb{C}^{(n-k )\times k}$ such that 
\begin{equation*}
dw = \sqrt{I_{n-k} - ww^*} dB \sqrt{I_k-w^*w},
\end{equation*}
whence
\begin{equation}\label{eq-VdV}
(dw^*)w + w^*(dw) = \sqrt{I_k-w^*w} dB^* \sqrt{I_{n-k} - ww^*}w  + w^*\sqrt{I_{n-k} - ww^*} dB \sqrt{I_k-w^*w}. 
\end{equation}
Furthermore,
\begin{align*}
w^*\sqrt{I_{n-k} - ww^*} dB \sqrt{I_k-w^*w} & = \sqrt{J} Q^*\sqrt{I_{n-k} - QJQ^*}dB \sqrt{I_k-w^*w}
\\& = \sqrt{J} \sqrt{I_{k} - J}Q^*dB\sqrt{I_k-J} , 
\end{align*}
and similarly 
\begin{align*}
\sqrt{I_k-w^*w} dB^* \sqrt{I_{n-k} - ww^*}w =  \sqrt{I_k-J} dB^* Q\sqrt{I_{k} - J} \sqrt{J}. 
\end{align*}
Finally, the identity $QQ^* = I_{n-k}$ shows that the matrix-valued process $dR := dB^*Q$ is a Brownian motion process in $\C^{k\times k}$. Altogether, we end up with the following autonomous SDE:
\begin{equation}\label{SDE-Matrix2}
dJ= \sqrt{I_k-J_t} dR\sqrt{I_{k} - J} \sqrt{J} + \sqrt{J} \sqrt{I_{k} - J}dR^{*}\sqrt{I_k-J} + 2((n-k) - \tr(J))(I_k - J) dt,
\end{equation}
as desired. 
\end{proof}
Using Theorem 4 in \cite{Gra-Mal}, we immediately obtain the SDE satisfied by the eigenvalues process $(\lambda_j)_{j=1}^k$ of $J$.
\begin{corollary}
Assume $\lambda_1(0) > \dots > \lambda_k(0) > 0$ and let 
\begin{equation*}
\tau := \inf\{t > 0, \lambda_l(t) = \lambda_j(t)\, \textrm{for some} \, (l,j) \},
\end{equation*}
be the first collision time. Then, for any $1 \leq j \leq k$:
\begin{multline}\label{eq-SDE-lambda}
d\lambda_j = 2\sqrt{\lambda_j}(1-\lambda_j) dN_j + 2\left[(n-k) - \sum_{l=1}^k\lambda_l\right]  {{(1-\lambda_j)}} dt \\ 
+ 2\sum_{l\neq j} \frac{(1-\lambda_l)\lambda_j(1-\lambda_j) + (1-\lambda_j)\lambda_l(1-\lambda_l)}{\lambda_j - \lambda_l} dt
\end{multline}
where $(N_j)_{j=1}^k$ is a $k$-dimensional Brownian motion. 
\end{corollary}
%
%
%

Using the rational transformation,

\[
\rho_j:=\frac{1+\lambda_j}{1-\lambda_j} >1 \quad \Leftrightarrow \quad \lambda_j = \frac{\rho_j-1}{\rho_j+1},
\]
we can prove that: 
\begin{proposition}\label{Collision}
The first collision time $\tau$ is infinite almost surely. 
\end{proposition}
\begin{proof}
A straightforward application of It\^o's formula yields: 
\begin{align*}
d\rho_j &= 4\frac{\sqrt{\lambda_j}}{(1-\lambda_j)} dN_j + 4\left[(n-1) - \lambda_j\right] \frac{dt}{1-\lambda_j}  + 8 \sum_{l\neq j} \frac{\lambda_l}{\lambda_j - \lambda_l} dt + 8 \frac{\lambda_j}{1-\lambda_j} dt
\\& = 2\sqrt{\rho_j^2-1}dN_j + 2\left[n\rho_j + (n-2) \right] dt + 4(\rho_j+1)\sum_{l\neq j} \frac{(\rho_l-1)}{(\rho_j-\rho_l)} dt
\\& = 2\sqrt{\rho_j^2-1}dN_j + 2\left[(n+2-2k)\rho_j + (n-2k) \right] dt  +4(\rho_j^2-1)\sum_{l\neq j} \frac{1}{(\rho_j-\rho_l)} dt.
\end{align*}
It follows that the generator of the diffusion process $(\rho_j)_{j=1}^k$ acts on smooth functions as: 
\begin{equation*}
L^{(k,n)}_{\rho_1, \dots, \rho_k} = \sum_{j=1}^k\left\{2(\rho_j^2-1)\partial_{j}^2 + 2\left[(n+2-2k)\rho_j + n-2k \right]\partial_j\right\} + 4 \sum_{l,j = 1, l \neq j}^k \frac{(\rho_j^2-1)}{(\rho_j-\rho_l)}\partial_j.
\end{equation*}
Up to a constant, this operator is the Vandermonde transform of $k$ independent diffusions with generator 
\begin{equation*}
\mathscr{L}^{(k,n)}_u := 2(u^2-1)\partial_{u}^2 + 2\left[(n+2-2k)u + n-2k \right]\partial_u.
\end{equation*}
Indeed, the Vandermonde function 
\begin{equation*}
V(\rho_1, \dots, \rho_k) = \prod_{1 \leq l < j \leq k} (\rho_l - \rho_j),  
\end{equation*} 
is positive on the Weyl chamber $\{\rho_1 > \rho_2 > \dots > \rho_k\}$ and satisfies (see Appendix in \cite{Dou}): 
\begin{equation*}
\sum_{j=1}^k\mathscr{L}^{(k,n)}_{\rho_j}V = \frac{k(k-1)(3n+2-4k)}{3}V.
\end{equation*}
Besides, one readily checks that 
\begin{equation}\label{DoobV}
L^{(k,n)}_{\rho_1, \dots, \rho_k}  = \frac{1}{V}\sum_{j=1}^k\mathscr{L}^{(k,n)}_{\rho_j}(V\cdot) - \frac{k(k-1)(3n+2-4k)}{3}.
\end{equation}
In particular, the process 
\begin{equation*}
t \mapsto \frac{1}{V(\rho_1(t), \dots, \rho_k(t))} + \frac{k(k-1)(3n+2-4k)}{3}, \quad t < \tau,  
\end{equation*}
starting at non-colliding particles is a continuous local martingale which blows-up as $t \rightarrow \tau$. Since it is a time-changed Brownian motion, the result follows (this is the classical McKean's argument). 
\end{proof}

We can further relate $(\rho_j)_{j=1}^k$ to a special instance of the so-called radial Heckman-Opdam process associated with the root system of type $BC_k$ (\cite{Sch}). More precisely, the process $(\zeta_j(t), t \geq 0)_{j=1}^k$ defined by: 
\begin{equation*}
\zeta_j(t) = \cosh^{-1}(\rho_j(t/4)), \quad 1 \leq j \leq k,
 \end{equation*}
satisfies
\begin{equation}\label{eq-SDE-coth}
d\zeta_j = dN_j + \frac{1}{2}\coth(\zeta_j)dt + \frac{(n-2k)}{2} \coth\left(\frac{\zeta_j}{2}\right)dt + \frac{1}{2}\sum_{i \neq j}\left[\coth\left(\frac{\zeta_i -\zeta_j}{2}\right) + \coth\left(\frac{\zeta_i + \zeta_j}{2}\right)\right]dt. 
\end{equation}
As a matter of fact, $(\rho_j)_{j=1}^k$ is the unique strong solution of the SDE it satisfies for any starting point $\rho(0) \in [1, \infty)^k$ (\cite{Sch}, Proposition 4.1). By the virtue  of Proposition \ref{Collision}, we deduce that 
\begin{corollary}
The SDE \ref{eq-SDE-lambda} admits a unique strong solution for all $t \geq 0$ and any starting point $\lambda_1(0) \geq \lambda_2(0) \geq \dots \geq \lambda_k(0) \geq 0$.
\end{corollary}

\section{Long-time behavior and distribution of the functional $\int_0^t \mathop{\tr} (J_s)ds$} 
In this section, we derive the limit as $t \rightarrow \infty$ of the functional: 
\begin{equation*}
\int_0^t \mathop{\tr} (J_s)ds,
\end{equation*}
and compute its Laplace transform. Indeed, the asymptotics of this functional may be readily derived from \eqref{eq-SDE-coth}. 

\begin{proposition}\label{Lemma1}
As $t\to+\infty$, almost surely we have that
\begin{equation}\label{eq-tr-J}
\lim_{t\to+\infty}\frac1t\int_0^t\tr(J_s)ds=k.
\end{equation}
\end{proposition}

\begin{proof}
Recall the SDE \eqref{eq-SDE-coth}:
\[
d\zeta_j = dN_j + \left(\coth(\zeta_j) + \frac{(n-2k)}{2} \coth\left(\frac{\zeta_j}{2}\right)\right)dt + \frac{1}{2}\sum_{i\not=j}\left[\coth\left(\frac{\zeta_i -\zeta_j}{2}\right) + \coth\left(\frac{\zeta_i + \zeta_j}{2}\right)\right]dt. 
\]
Under the standing assumption of $n-k\ge k$ and using a comparison argument (see e.g. Proposition 4.2 in \cite{Sch}), we infer that $\zeta_j\to +\infty$ almost surely as $t\to+\infty$. Therefore 
\[
\rho_j(t)=\cosh(\zeta_j(4t))\to +\infty,\quad j=1,\dots, k,
\]
and in turn it holds almost surely that 
\[
\lambda_j(t)=\frac{\rho_j-1}{\rho_j+1}\to 1,\quad j=1,\dots, k.
\]
The conclusion follows.
\end{proof}

Now, we shall get more insight into the distribution of the above functional and give an expression of its Laplace transform relying on Girsanov Theorem and Karlin-McGregor formula. Moreover, as in \cite{Demni}, we shall further point out an interesting connection with the generalized Maass Laplacian in the complex hyperbolic space (\cite{Aya-Int}). 

Let $(w_t)_{t \ge 0}=(X_t Z_t^{-1})_{t \ge 0}$ be a Brownian motion on $HG_{n,k}$ as in Theorem \ref{main:s1} and recall $J=w^*w$ as well its the eigenvalues process $(\lambda_j)_{j=1}^k$ of $J$. Recall also
\begin{equation*}
\rho_j:=\frac{1+\lambda_j}{1-\lambda_j}.
\end{equation*}
We assume that the $\lambda_j(0)$'s and therefore the $\rho_j(0)$'s are pairwise distinct. This is not a loss of generality and our results extend by continuity to non necessarily pairwise distinct eigenvalues.

\subsection{An auxiliary lemma}
Let 
$$\Delta_k=\{ \rho \in \mathbb{R}^k, \rho_1 > \rho_2 > \dots > \rho_k>1\}$$ and for any $\alpha \geq 0$, introduce the following diffusion operator: 
\begin{equation}\label{Gen1}
\mathscr{L}_u^{(k,n,\alpha)} := 2(u^2-1)\partial_{u}^2 + 2\left[(n+2-2k+2\alpha)u + n-2k -2\alpha \right]\partial_u, \quad u \geq 1. 
\end{equation}
 Performing the variable change $u = \cosh(2r), r \geq 0$, $\mathscr{L}_u^{(k,n,\alpha)}$ is mapped into the following hyperbolic Jacobi operator (\cite{Koor}): 
 \begin{equation*}
\mathscr{H}_r^{(n,k, \alpha)} =  \frac{1}{2}\left\{\partial_r^2 + \left[(2(n-2k)+1)\coth\left(r\right) + (4\alpha + 1)\tanh\left(r\right)\right]\partial_r\right\}. 
 \end{equation*}
In particular, $\mathscr{L}_u^{(k,n,0)} = \mathscr{L}_u^{(k,n)}$ is mapped into the radial part of the Laplace-Beltrami operator on the complex hyperbolic space of $H_{n-2k+1} \sim SU(n-2k+1,1)/SU(n-2k+1)$. 

Let $q_t^{(n,k,\alpha)}$ denote the heat kernel (with respect to Lebesgue measure) of $\mathscr{L}_u^{(k,n,\alpha)}$ with Neumann boundary condition at $u=1$. This kernel does not admits in general a simple expression as can be seen from the Jacobi-Fourier inversion formula (Theorem 2.3. in \cite{Koor}). Nonetheless, this formula simplifies when $\alpha = 0$ (see also \cite{Ner}, \cite{Sch}) and yields: 
\begin{multline*}
q_t^{(n,k,0)}(u_1, u_2) = \frac{(u_2-1)^{n-2k}}{\pi 2^{n-2k},}\int_0^{\infty}e^{-2t(\mu^2+\kappa^2)}
F_{-\mu}(-v_1)F_{\mu}(-v_2) \left|\frac{\Gamma(\kappa+i\mu)\Gamma(\kappa-2\alpha + i\mu)}{\Gamma(2i\mu)}\right|^2 d\mu  
\end{multline*}
 where $v_i = (u_i-1)/2, i \in \{1,2\}$, $\kappa:= (n-2k+1)/2$ and 
\begin{equation*}
F_{\mu}(-v_i) = {}_2F_1\left(\kappa+ i\mu, \kappa-i\mu, n-2k+1; \frac{1-u_i}{2}\right), \quad i \in \{1,2\}, \quad \mu \in \mathbb{R},
\end{equation*}
are Jacobi functions. In this case, we have: 
\begin{lemma}\label{Estimate}
For any $u_1 \geq 1, t > 0, \alpha \geq 0$,  
 \begin{equation*}
\int_1^{\infty} (u_2)^{\alpha} q_t^{(n,k,0)}(u_1, u_2) du_2 < \infty. 
\end{equation*}
Consequently, 
\begin{equation*}
\mathbb{E}\left[ \sup_{0 \le s \le t} \det(I_k-J(s))^{-\alpha}\right] < \infty.
\end{equation*}
\end{lemma}
\begin{proof}

As to the first assertion, we use the comparison principle for stochastic differential equations together with the obvious inequality $\tanh(r) < \coth(r)$ to see that the diffusion associated to $\mathscr{H}_r^{(n,k,0)}$ is dominated by the (unique strong) solution 
$(H(t))_{t \geq 0}$ of the SDE: 
\begin{equation*}
H(t) = H(0) + \gamma(t) + (n-2k+1)\int_0^t \coth(H(s)) ds,
\end{equation*}
where $\gamma$ is a real Brownian motion. The diffusion $(\cosh(H(t)))_{t \geq 0}$ is studied in detail for instance in \cite{Jak-Wis} from which it is seen that it has moments of all orders. Keeping in mind the correspondence between $\mathscr{L}_u^{(k,n,\alpha)}$ and 
$\mathscr{H}_r^{(n,k, \alpha)}$, the first part of the lemma follows.

For the second part of the lemma, without loss of generality we assume $\rho(0) \in \Delta_k$. We then first note that from the non-collision property of $\rho$.
\begin{align*}
\det(I_k-J(t))^{-\alpha} & =2^{-\alpha k} \ \prod_{j=1}^k(1+\rho_j(t))^{\alpha} \\
 & \le 2^{-\alpha k} \prod_{j=1}^k(1+\rho_1(t))^{\alpha}  \\
 & \le 2^{-\alpha k} (1+\rho_1(t))^{k \alpha}
\end{align*}
It is therefore enough to prove that for every $\alpha \ge 0$ and $t \ge 0$ we have
$
\mathbb{E} \left( \sup_{0\le s \le t} \rho_1(s)^\alpha \right) < \infty
$
The SDE satisfied by $\rho_1$ has a non-negative drift and therefore the process $\rho_1(t)$ is a sub-martingale. From Doob's maximal inequality, it is therefore enough to prove that  for every $\alpha \ge 0$ and $t \ge 0$ we have
$
\mathbb{E} \left( \rho_1(t)^\alpha \right) < \infty.
$

From  \eqref{DoobV} and Karlin-McGregor formula the semigroup density of $(\rho_j(t))_{j=1}^k$ may be written as: 
\begin{equation*}
e^{-k(k-1)(3n+2-4k)t/3} \frac{V(\rho)}{V(\rho(0))} \det \left(q_t^{(n,k,0)}(\rho_j(0), \rho_a)\right)_{j,a = 1}^k. 
\end{equation*}
The conclusion follows therefore from the first part of the lemma.
 \end{proof}

\subsection{Laplace transform of $\int_0^t \mathop{\tr} (J_s)ds$}
Now, we are ready to prove the following theorem:
\begin{theorem}\label{theo-Laplace}
Assume $\rho(0) \in \Delta_k$. For every $\alpha >0$ and $t >0$,
\begin{align*}
 & \mathbb{E}\left\{\exp-\left(2\alpha^2 \int_0^t \mathop{\tr}(J(s)) ds \right)\right\} \\
 = & \frac{e^{[6\alpha (n-2k+1)-(k-1)(3n+2-4k)]kt/3}}{V(\rho(0))} \left[ \prod_{j=1}^k(1+\rho_j(0))^{\alpha} \right]
\det\left(\int_{1}^{+\infty} \frac{u^{a-1}}{(1+u)^{\alpha}}q_t^{(n,k,\alpha)}(\rho_j(0), u) du\right)_{a,j=1}^k. 
\end{align*}
\end{theorem}

\begin{proof}
The proof is rather long and adapts to the hyperbolic case ideas developed in \cite{BW-Winding}. The main ingredient is a matrix Girsanov transform.
Recall first the SDE \eqref{eq-SDE-lambda} for $(\lambda_j)_{j=1}^k$. It can be simplified to
\begin{equation*}
d\lambda_j =  2\sqrt{\lambda_j}(1-\lambda_j) dN_j + 2\left[(n-1) - \lambda_j\right] (1-\lambda_j) dt  + 4(1-\lambda_j)^2 \sum_{l\neq j} \frac{\lambda_l}{\lambda_j - \lambda_l} dt, 
\end{equation*}
for $1 \leq j \leq k$. Then 
\begin{equation*}
-d\log(1-\lambda_j) =  2\sqrt{\lambda_j} dN_j + 2(n-1) dt  + 4(1-\lambda_j) \sum_{l\neq j} \frac{\lambda_l}{\lambda_j - \lambda_l} dt, 
\end{equation*}
and in turn 
\begin{align*}
- d \log\det(I_k-J) & = 2\sum_{j=1}^k\sqrt{\lambda_j} dN_j  + 2(n-1)kdt + 4\sum_{j=1}^k \sum_{l\neq j} (1-\lambda_j)  \frac{\lambda_l}{\lambda_j - \lambda_l} dt 
\\& = 2\sum_{j=1}^k\sqrt{\lambda_j} dN_j  + 2(n-1)kdt + 4 \sum_{j=1}^k \sum_{l\neq j}   \frac{\lambda_l}{\lambda_j - \lambda_l} dt 
\\& = 2\sum_{j=1}^k\sqrt{\lambda_j} dN_j  + 2k(n-k) dt. 
\end{align*}
Consequently, for any $\alpha > 0$, the exponential local martingale 
\begin{equation*}
M_t^{(\alpha)}:= \exp\left(2\alpha\int_0^t \sum_{j=1}^k\sqrt{\lambda_j}(s) dN_j(s) - 2\alpha^2 \int_0^t \mathop{\tr}(J(s)) ds \right), \ t\ge0
\end{equation*}
may be written as 
\begin{equation*}
M_t^{(\alpha)} = e^{-2\alpha k(n-k)t} \left[\frac{\det(I_k-J(0))}{\det(I_k-J(t))}\right]^{\alpha} \exp-\left(2\alpha^2 \int_0^t \mathop{\tr}(J(s)) ds \right), \ t\ge0.
\end{equation*}
$(M_t^{(\alpha)})_{t \geq 0}$ is also a martingale since 
\begin{equation*}
M_t^{(\alpha)} \leq \left[\det(I_k-J(t))\right]^{-\alpha}
\end{equation*}
and by the virtue of Lemma \ref{Estimate}. 
We can therefore define for any fixed time $t > 0$,  a new probability measure 
\begin{equation*}
P^{(\alpha)}_{|\mathscr{F}_t} = M_t^{(\alpha)} P_{|\mathscr{F}_t}
\end{equation*}
and denote $\mathbb{E}^{(\alpha)}$ the corresponding expectation. Then under $P^{(\alpha)}$
\begin{equation*}
\tilde{N}_j(t) := N_j(t) - 2\alpha\int_0^t\sqrt{\lambda_j(s)}ds, \quad 1 \leq j \leq k,
\end{equation*}
defines a $k$-dimensional Brownian motion so that 
\begin{equation*}
d\lambda_j =  2\sqrt{\lambda_j}(1-\lambda_j) d\tilde{N}_j + 2\left[(n-1) - (1-2\alpha)\lambda_j\right] (1-\lambda_j) dt  + 4(1-\lambda_j)^2 \sum_{l\neq j} \frac{\lambda_l}{\lambda_j - \lambda_l} dt. 
\end{equation*}
Recalling 
\[
\rho_j=\frac{1+\lambda_j}{1-\lambda_j},
\]
then It\^o's formula yields: 
\begin{align*}
d\rho_j &= 
 2\sqrt{\rho_j^2-1}d\tilde{N}_j + 2\left[(n+2-2k+2\alpha)\rho_j + n-2\alpha-2k \right] dt  +4(\rho_j^2-1)\sum_{l\neq j} \frac{1}{(\rho_j-\rho_l)} dt.
\end{align*}
Up to the multiplication operator by the constant 
\begin{equation*}
\frac{k(k-1)(3n+2-4k+6\alpha)}{3}, 
\end{equation*}
this process is again a Vandermonde transform of $k$ independent copies of the diffusion whose generator is given by: 
\begin{equation}\label{Gen2}
2(u^2-1)\partial_{u}^2 + 2\left[(n+2-2k+2\alpha)u + n-2k -2\alpha \right]\partial_u, \quad u \geq 1,
\end{equation}
with Neumann boundary condition at $u=1$. As a result, Karlin-McGregor's formula entails (see e.g. \cite{AOW} and references therein):  
\begin{align*}
\mathbb{E}\left\{\exp-\left(2\alpha^2 \int_0^t \mathop{\tr}(J(s)) ds \right)\right\} & = e^{2\alpha k(n-k)t} \mathbb{E}^{(\alpha)}\left[\frac{\det(I_k-J(t))}{\det(I_k-J(0))}\right]^{\alpha} 
\\& = e^{2\alpha k(n-k)t}\prod_{j=1}^k(1+\rho_j(0))^{\alpha}  \mathbb{E}^{(\alpha)}\left[\prod_{j=1}^k\frac{1}{(1+\rho_j(t))^{\alpha}}\right]
\\& = e^{[6\alpha (n-2k+1)-(k-1)(3n+2-4k)]kt/3}\prod_{j=1}^k(1+\rho_j(0))^{\alpha} 
\\& \int_{\Delta_k} \prod_{j=1}^k\frac{d\rho_j}{(1+\rho_j)^{\alpha}} \frac{V(\rho)}{V(\rho(0))} \det(q_t^{(n,k,\alpha)}(\rho_a(0), \rho_j))_{a,j=1}^k
\end{align*}
where $q_t^{(\alpha,k)}(u, v)$ is the heat semi-group of the infinitesimal generator \eqref{Gen1}. Equivalently, the Andr\'eief identity entails (\cite{Dei-Gio}, p. 37) 
\begin{multline*}
\mathbb{E}\left\{\exp-\left(2\alpha^2 \int_0^t \mathop{\tr}(J(s)) ds \right)\right\} = \frac{e^{[6\alpha (n-2k+1)-(k-1)(3n+2-4k)]kt/3}}{V(\rho(0))}\prod_{j=1}^k(1+\rho_j(0))^{\alpha} \\ 
\det\left(\int \frac{u^{a-1}}{(1+u)^{\alpha}}q_t^{(\alpha,k)}(\rho_j(0), u) du\right)_{a,j=1}^k. 
\end{multline*}
\end{proof}

\subsection{Connection to the Maass Laplacian in the complex hyperbolic space}
In this paragraph, we present the connection of $q_t^{(n,k,\alpha)}$ to heat semi-group of the Maass Laplacian on the complex hyperbolic space $\mathbb{C} H^{n-2k+1}$ realized in the unit ball (\cite{Aya-Int}). For $k=1$, such connection was already pointed out in \cite{Demni} and was the key ingredient to derive the density of the corresponding stochastic area. For higher ranks, the computations become tedious. Nonetheless, as in \cite{Demni}, the new expression we obtain below makes transparent the limiting behavior proved in \ref{Lemma1} through the exponential factor $e^{-2\alpha^2 t}$ and is somehow more explicit than the one displayed in Theorem \ref{theo-Laplace} since the heat semi-group of the Maass Laplacian admits a more compact form than $q_t^{(n,k,\alpha)}$ (\cite{Aya-Int}, Theorem 2.2).

Firstly, we perform the variable change $u = \cosh(2r)$ and use the identity $2\cosh^2(r) = 1+\cosh(2r)$ to get: 
\begin{multline}\label{eq-int-trace}
\mathbb{E}\left\{\exp-\left(2\alpha^2 \int_0^t \mathop{\tr}(J(s)) ds \right)\right\} = \frac{e^{[6\alpha (n-2k+1)-(k-1)(3n+2-4k)]kt/3}}{V(\rho(0))} \prod_{j=1}^k(1+\rho_j(0))^{\alpha} \\ 
 \det\left(\frac{[\cosh(2r)]^{a-1}}{2^{\alpha}[\cosh(r)]^{2\alpha}}q_t^{(n,k,\alpha)}(\rho_j(0), \cosh(2r)) d(\cosh(2r))\right)_{a,j=1}^k. 
\end{multline}
On the other hand, from \cite{Demni} (see the proof of Theorem 1), we infer that the hyperbolic Jacobi operator $\mathscr{H}_r^{(n,k)}$ is intertwined via the map: 
\begin{equation*}
f \mapsto \frac{1}{\cosh^{2\alpha}(r)}f
\end{equation*}
with the radial part of the shifted Maass Laplacian $\mathcal{L}$ in the complex hyperbolic space $\mathbb{C} H^{n-2k+1}$. More precisely, it holds that
\[
\mathscr{H}_r^{(n,k)}\left(r \mapsto \frac{1}{\cosh^{2\alpha}r}f(r) \right) + \frac{\left(2\alpha+{n-2k+1}\right)^2}{2\cosh^{2\alpha}r}f(r) =\frac{1}{\cosh^{2\alpha}(r)}\mathcal{L}(f)(r)
\]
where
\begin{equation}\label{Radial}
\mathcal{L}=\frac{1}{2}\partial_{r}^2 + \left[\left(n-2k+\frac{1}{2}\right)\coth(r) + \frac{1}{2}\tanh(r) \right]\partial_{r} + \frac{2\alpha^2}{\cosh^2(r)} + \frac{(n-2k+1)^2}{2}.
 \end{equation}
Consequently, the following identity holds: 
\begin{multline}\label{eq-kernel-identity}
\frac{e^{2\alpha (n-2k+1)t}}{[\cosh(r)]^{2\alpha}} q_t^{(\alpha,k)}(\rho_j(0), \cosh(2r)) d(\cosh(2r)) = \frac{2^{\alpha} e^{-2\alpha^2t}}{(1+\rho_j(0))^{\alpha}} \\ 
v_t^{(n-2k+1, \alpha)}\left(\frac{1}{2}\cosh^{-1}(\rho_j(0)),r\right)dr,
\end{multline}
where $v_t^{(n-2k+1, \alpha)}$ is the heat kernel of $\mathcal{L} - (n-2k+1)^2/2$ in \eqref{Radial} with respect to the radial volume element (the numerical factor was missed in \cite{Demni} and is simply the volume of the Euclidean sphere in $\mathbb{C}^{n-2k+1}$): 
\begin{equation*}
\frac{2\pi^{n-2k+1}}{\Gamma(n-2k+1)}(\sinh(r))^{2(n-2k+1)-1}\cosh(r).
\end{equation*}
By plugging \eqref{eq-kernel-identity} into \eqref{eq-int-trace}, we arrive at: 
\begin{proposition}
The Laplace transform derived in Theorem \ref{theo-Laplace} can be rewritten as: 
\begin{multline}\label{NewLaplace}
\mathbb{E}\left\{\exp-\left(2\alpha^2 \int_0^t \mathop{\tr}(J(s)) ds \right)\right\} = \frac{e^{-2\alpha^2kt}}{V(\rho(0))} e^{-(k-1)(3n+2-4k)kt/3}\frac{2\pi^{n-2k+1}}{\Gamma(n-2k+1)} \\
\det\left(\int [\cosh(2\zeta)]^{a-1}v_t^{(n-2k+1, \alpha)}\left(\frac{1}{2}\cosh^{-1}(\rho_j(0)),r\right) (\sinh(r))^{2(n-2k+1)-1}\cosh(r) dr\right)_{a,j=1}^k. 
\end{multline}
\end{proposition}
\begin{remark} 
Note the existing shift between the complex dimension $n \geq 1$ in \cite{Demni} and the rank-one case $k=1$ here corresponding the complex dimension $n-1, n \geq 2$.
\end{remark}

\begin{remark}
The heat kernel $v_t^{(n-2k+1, \alpha)}(0,\cosh\zeta)$ is given by (\cite{Aya-Int}, Theorem 2.2.): 
\begin{multline*}
 v_t^{(n-2k+1, \alpha)}\left(0, r\right) = \frac{4\pi^{n-2k+1}}{\Gamma(n-2k+1)} \int_{r}^{\infty} \frac{dx \sinh(x)}{\sqrt{\cosh^2(x) - \cosh^2(r)}} s_{t,2(n-2k+1)+1}(\cosh(x))
 \\ {}_2F_1\left(-2\alpha, 2\alpha, \frac{1}{2}; \frac{\cosh(r) - \cosh(x)}{2\cosh(r)}\right)dx 
\end{multline*}
where ${}_2F_1$ is the Gauss hypergeometric function and 
\begin{align*}
s_{t,2(n-2k+1)+1}(\cosh(x)) = \frac{e^{-(n-2k+1)^2t/2}}{(2\pi)^{n-2k+1}\sqrt{2\pi t}}\left(-\frac{1}{\sinh(x)}\frac{d}{dx}\right)^{n-2k+1}e^{-x^2/(2t)}, \quad x > 0,
\end{align*}
is the heat kernel with respect to the volume measure of the $2(n-2k+1)+1$-dimensional real hyperbolic space $H^{2n+1}$. 

More generally, if $r_0 \geq 0$ then the heat kernel $v_t^{(n-2k+1, \alpha)}(r_0,r)$ comes with the following additional term resulting from the integration over the sphere $S^{2n-1}$ of an automorphy factor (see \cite{Aya-Int}, Theorem 2.2):
\begin{equation*}
e^{-2i\alpha \arg(1-\langle z,w\rangle)}, \quad |z| = \tanh(r_0),\quad |w| = \tanh(r), \quad w,z \in \mathbb{C} H^{n-2k+1},
\end{equation*}
with respect to the angular part of $w$. 
\end{remark}
\begin{remark}
In \eqref{NewLaplace} appears the factor $e^{-2\alpha^2t}$ which governs the limiting behavior of 
\begin{equation*}
\int_0^t \mathop{\tr} (J_s)ds
\end{equation*}
after rescaling $\alpha \rightarrow \alpha/\sqrt{2t}$. It is then tempting and interesting to deduce Proposition \ref{Lemma1} from  \eqref{NewLaplace}. For $k=1$, this guess was announced in \cite{Demni} without proof and we shall revisit this case in the appendix. In particular, we obtain the whole expansion of the Laplace transform in the time $t$-variable for small values of $\alpha$. Such expansion seems out of reach for the moment since the computations are already tedious and tricky in the rank-one case. 
\end{remark}

\section{Skew-product decomposition, generalized stochastic areas and asymptotic windings}

\subsection{Skew-product decomposition}

Recall from Section \ref{sec:notations} the hyperbolic Stiefel fibration 
\[
\mathbf{U}(k) \to HV_{n,k} \to HG_{n,k},
\]
from which we can view $HV_{n,k}$ as a $\mathbf{U}(k)$-principal bundle over $HG_{n,k}$. Our goal in this section is to decompose the Brownian motion in $HV_{n,k}$ as a skew-product with respect to this fibration.

We first note that a computation similar to the computation done in \cite[Lemma 3.1]{BW-Winding} shows that the connection form of this bundle is given by the $\mathfrak{u}(k)$-valued one-form defined on $HV_{n,k}$ by
\begin{align}\label{eq-contact-form}
\omega:&=\frac{1}{2} \left( (X^* \,  Z^*)\begin{pmatrix} I_{n-k}   & 0 \\ 0 & -I_{k} \end{pmatrix}d\begin{pmatrix}  X \\ Z  \end{pmatrix}-d(X^* \, Z^*)\begin{pmatrix} I_{n-k}   & 0 \\ 0 & -I_{k} \end{pmatrix}\begin{pmatrix}  X \\ Z  \end{pmatrix}\right) \\
&=\frac{1}{2} \left(X^*dX-dX^*X -(Z^*dZ-dZ^*Z)\right).
\end{align}

In the inhomogeneous coordinate $w:=XZ^{-1}$ introduced in Section \ref{sec:notations}, we consider then the following $\mathfrak{u}(k)$ valued one-form defined on $HG_{n,k}$ 
 \begin{align}\label{eq-def-eta}
 \eta:=& \frac12 \left( (I_k-w^* w)^{-1/2} (w^*dw-dw^* \, w)(I_k-w^* w)^{-1/2}  \right. \\
  & \left. - (I_k-w^* w)^{-1/2}\, d(I_k-w^* w)^{1/2}+d(I_k-w^* w)^{1/2} \, (I_k-w^* w)^{-1/2} \right) \notag
 \end{align}

We  are now in position to prove the following skew-product decomposition of the Brownian motion on ${HG}_{n,k}$.
\begin{theorem}\label{skew}
Let $(w_t)_{t \ge 0}$ be a Brownian motion on ${HG}_{n,k}$ started at $w_0=X_0Z_0^{-1} \in {HG}_{n,k}$ and $(\Theta_t)_{t \ge 0}$ the $\mathbf{U}(k)$-valued  solution of the Stratonovich stochastic differential equation
\begin{align*}
\begin{cases}
d\Theta_t = \circ d\mathfrak a_t \, \Theta_t \\
\Theta_0=(Z_0 Z^*_0)^{-1/2}Z_0,
\end{cases}
\end{align*}
where  $\mathfrak{a}_t= \int_{w[0,t]} \eta$. Then the process
\[
\widetilde{w}_t:=\begin{pmatrix} w_t  \\  I_k  \end{pmatrix}(I_k-w_t^*w_t)^{-1/2}\Theta_t
\]
is the horizontal lift at $\begin{pmatrix} X_0  \\ Z_0 \end{pmatrix}$ of $(w_t)_{t \ge 0}$ to ${HV}_{n,k}$. Moreover, if we denote by $(\Omega_t)_{t \ge 0}$ a Brownian motion on the unitary group $\mathbf{U}(k)$ independent from $(w_t)_{t \ge 0}$, then the process $$\begin{pmatrix} w_t  \\  I_k  \end{pmatrix}(I_k-w_t^*w_t)^{-1/2}\Theta_t \,  \Omega_t$$
is a Brownian motion on $HV_{n,k}$ started at $\begin{pmatrix} X_0  \\ Z_0 \end{pmatrix}$.
\end{theorem}
\begin{proof}
Note the fact that on $\widetilde{w}[0,t]$,
\begin{align*}
\omega&=\frac12\left(\circ d\Theta^*\Theta-\Theta^*\circ d\Theta \right)+\Theta^*\circ \eta\Theta\\
&=\Theta^*\circ\left(-d\mathfrak{a}+\eta\right) \Theta=0,
\end{align*}
where $\Theta=(I_k-w^*w)^{1/2}Z$.
The first assertion then follows from the definition of horizontal stochastic lift, namely 
$$\int_{\widetilde{w}[0,t]} \omega=0.$$

By taking into account that the pseudo-Riemannian submersion $HV_{n,k}\to {HG}_{n,k}$ is totally geodesic, the second assertion then follows from an analogues argument as in the proof of Theorem 3.3 in the spherical case (\cite{BW-Winding}).
\end{proof}

\subsection{Limit theorem for the generalized stochastic area process in the hyperbolic complex Grassmannian}

Let $(w_t)_{t \ge 0}$ be a Brownian motion on $HG_{n,k}$ as in Theorem \ref{main:s1}. From \eqref{eq-def-eta} we have that
\begin{align}\label{eq-int-tr-eta}
-\int_{w[0,t]} \mathrm{tr} (\eta) & =\frac12 \mathrm{tr} \left[ \int_0^t   (I_k-J)^{-1/2} (\circ dw^* \, w-w^*\circ dw)(I_k-J)^{-1/2}\right] \notag\\
 & =\frac12 \mathrm{tr} \left[ \int_0^t   (I_k-J)^{-1/2} ( dw^* \, w-w^* dw)(I_k-J)^{-1/2}\right].
\end{align}
Using the complex Hermitian version of Lemma 1 in \cite{BZ} (see Appendix B below), we can show after simple, yet lengthy, computations that (see \cite{Won}, p. 593, for the compact analogue):
\begin{align*}
d \mathrm{tr} (\eta)  & = \mathrm{tr}\left[(I_k -w^*w)^{-1}(dw)^*(I_{n-k} - ww^*)^{-1} dw\right] 
\\& =-\partial \overline{\partial} \ln \det (I_k -w^*w) 
\end{align*}
where 
\begin{equation*}
\partial: = \sum_{1 \leq a \leq n-k, 1 \leq b\leq k} dw_{a,b}\frac{\partial}{\partial{w_{a,b}}}, \quad \overline{\partial}: = \sum_{1 \leq a \leq n-k, 1 \leq b\leq k} d\overline{w}_{a,b}\frac{\partial}{\partial{\overline{w}_{a,b}}}
\end{equation*}
are the Dolbeaut operators and are such that $d = \partial + \overline{\partial}$. Of course, the product of two matrix-valued one-forms is the exterior product of their entries. It follows that $i \mathrm{tr} (d\eta)$ is the K\"ahler form on the complex manifold $HG_{n,k}$. Therefore, in some sense, the functional $\int_{w[0,t]} \mathrm{tr} (\eta)$ can be interpreted as  a generalized  stochastic area process on $HG_{n,k}$ as in \cite{BW-SWHF}, and as a nice consequence of \eqref{eq-int-tr-eta}, we obtain  in the theorem below the limiting theorem for this generalized stochastic area process.

\begin{theorem}\label{LimitTheorem}
When $t\to+\infty$, we have in distribution 
\[
\frac{1}{i\sqrt{t}} \int_{w[0,t]} \mathrm{tr} (\eta) \to \mathcal{N}(0,k),
\]
\end{theorem}
\begin{proof}
From the proof of Proposition \ref{SDE-Matrix}, we readily have: 
\begin{equation*}
dw^*w - w^*dw =\sqrt{I_k-J_t} dR\sqrt{I_{k} - J} \sqrt{J} - \sqrt{J} \sqrt{I_{k} - J}dR^*\sqrt{I_k-J}, 
\end{equation*}
where $(R_t)_{t\ge0}$ is a $k\times k$ complex matrix-valued Brownian motion. Hence
\begin{align*}
 (I_k-J)^{-1/2} ( dw^* \, w-w^* dw)(I_k-J)^{-1/2}&=dR\sqrt{J}-\sqrt{J} dR^*\\
 &=dRU\sqrt{\Lambda}U^*-U\sqrt{\Lambda}U^*dR^*
\end{align*}
where the second quality comes from the diagonalization of $J=U\Lambda U^*$ where $U\in \mathbf{U}(k)$ and $\Lambda=\mathrm{diag}\{\lambda_1,\dots, \lambda_k\}$. Therefore
\begin{align*}
-\int_{w[0,t]} \mathrm{tr} (\eta) 
 & =\frac12  \int_0^t  \mathrm{tr} \left[ dR\sqrt{\Lambda}-\sqrt{\Lambda}dR^*\right]\\
 &\stackrel{D}{=}i\mathcal{B}_{\int_0^t \mathrm{tr}(J)ds}.
\end{align*}
where $\mathcal{B}$ is a one-dimensional Brownian motion independent of $J$.  We then obtain the desired conclusion  from \eqref{eq-tr-J}.
\end{proof}

\begin{remark}
When $k=1$, the limiting result stated in Theorem \ref{LimitTheorem} coincides with the one for the stochastic area process on the anti-de Sitter space (see Theorem 3.7 in \cite{BW-SWHF}). 
\end{remark}

\begin{remark}
The previous proof has shown that
\begin{align*}
\int_{w[0,t]} \mathrm{tr} (\eta) 
 &=i\mathcal{B}_{\int_0^t \mathrm{tr}(J)ds}.
\end{align*}
where $\mathcal{B}$ is a one-dimensional Brownian motion independent of $J$, therefore
\[
\mathbb{E} \left( e^{i \alpha \int_{w[0,t]} \mathrm{tr} (\eta) } \right)=\mathbb{E} \left( e^{- \alpha \mathcal{B}_{\int_0^t \mathrm{tr}(J)ds}} \right)=\mathbb{E} \left( e^{- \frac{\alpha^2}{2} \int_0^t \mathrm{tr}(J)ds} \right).
\]
A formula for the Laplace transform of the generalized area functional $\frac{1}{i} \int_{w[0,t]} \mathrm{tr} (\eta) $ therefore follows from  Theorem \ref{theo-Laplace}.
\end{remark}

\subsection{Limit theorem for the windings of the block determinants of the Brownian motion in $\mathbf{U}(n-k,k)$} 

In this section, we give an application of Theorems \ref{skew} and  \ref{LimitTheorem} to the study of the Brownian windings in the Lie group $\mathbf{U}(n-k,k)$.

\begin{theorem}
Let $U_t=\begin{pmatrix}  Y_t & X_t \\ W_t & Z_t \end{pmatrix}$ be a Brownian motion on the Lie group $\mathbf{U}(n-k,k)$ with $1 \le k \le n-k$.  One has the polar decomposition
\[
\det (Z_t)=\varrho_t e^{i\theta_t}
\]
where $\varrho_t  \ge 1$ is a continuous semimartingale  and $\theta_t$ is a continuous martingale such that the following convergence holds in distribution when $t \to +\infty$
\[
\frac{\theta_t}{\sqrt{t}} \to \mathcal{N}(0,2k) .
\]

\end{theorem}

 \begin{proof}
 We first note that from Theorem \ref{skew}, we have the identity in law
\[
\det (Z_t) =\det (I_k-w_t^*w_t)^{-1/2} \det \Theta_t \, \det  \Omega_t.
\]
From Lemma 4.6 in \cite{BW-Winding} we therefore obtain
\[
\det (Z_t) =\varrho_t e^{i \theta_t}
\]
with
\[
\varrho_t =\det(I_k-J_t)^{-1/2}, \, \, i \theta_t = i\theta_0+ \mathrm{tr}(D_t)+ \int_{w[0,t]} \mathrm{tr} (\eta) 
\]
where $D_t$ is a Brownian motion on $\mathfrak{u}(k)$ independent from $w$ and $\theta_0$ is such that 
\begin{equation*}
e^{i\theta_0}=\frac{\det Z_0}{| \det Z_0|}.
\end{equation*}
The result follows then from Theorem \ref{LimitTheorem} and from the scaling property of $\mathrm{tr}(D_t)$: 
\begin{equation*}
\mathrm{tr}(D_t) \overset{d}{=} i\sqrt{t} \mathcal{N}(0,k).
\end{equation*}
\end{proof}

\section{Appendix A: the full expansion of the Laplace transform in the rank-one case}
 For $k=1$, take $\rho(0) = 1$, then the Laplace transform \eqref{NewLaplace} reduces to: 
\begin{multline}\label{Laplace1}
\mathbb{E}\left\{\exp-\left(2\alpha^2 \int_0^t J(s) ds \right)\right\} = \frac{2\pi^{n-1}}{\Gamma(n-1)} e^{-2\alpha^2t} \\ \int_0^{\infty} v_t^{(n-1, \alpha)}\left(0,r\right) (\sinh(r))^{2n-3}\cosh(r) dr, 
\end{multline}
where (\cite{Aya-Int}, see also the proof of Theorem 1 in \cite{Demni}):
\begin{multline*}
 v_t^{(n-1, \alpha)}\left(0,r\right) = 2 \int_{r}^{\infty} \frac{dx \sinh(x)}{\sqrt{\cosh^2(x) - \cosh^2(r)}} s_{t,2(n-1)+1}(\cosh(x))
 \\ {}_2F_1\left(-2\alpha, 2\alpha, \frac{1}{2}; \frac{\cosh(r) - \cosh(x)}{2\cosh(r)}\right)dx. 
\end{multline*}
For ease of writing and in order to match our notations with those used in \cite{Demni}, we shall shift $n \geq 2 \rightarrow n+1, n \geq 1$. After all, the result of Proposition \ref{Lemma1} does not depend on $n$. 
In order to derive the limit of \eqref{Laplace1} as $t \rightarrow \infty$ (after rescaling), we firstly apply the quadratic transformation (see \cite{Erd}, (18), p. 112): 
\begin{equation*}
{}_2F_1(a,b,(a+b+1)/2; z) = {}_2F_1(a/2,b/2,(a+b+1)/2; 4z(1-z))
\end{equation*}
followed by Euler's transformation (\cite{Erd}, (22), p.64): 
\begin{equation*}
{}_2F_1(a,b,c; z) = (1-z)^{-a}{}_2F_1\left(a,c-b,c; \frac{z}{z-1}\right),
\end{equation*}
valid whenever both sides are analytic. Doing so, the heat semi-group $v_t^{(n, \alpha)}\left(0,r\right)$ may be written as: 
\begin{multline*}
v_t^{(n, \alpha)}\left(0,r\right) = \frac{4\pi^{n}}{\Gamma(n)} \int_{r}^{\infty} \frac{dx \sinh(x)}{\sqrt{\cosh^2(x) - \cosh^2(r)}} \frac{\cosh^{2\alpha}(r)}{\cosh^{2\alpha}(x)} 
\\ {}_2F_1\left(\frac{1}{2}+\alpha, \alpha, \frac{1}{2}; 1- \frac{\cosh^2(r)}{\cosh^2(x)}\right)  s_{t,2n+1}(\cosh(x))dx.
\end{multline*}
This expression has the merit to involve the Gauss hypergeometric series rather than its analytic continuation to the slit plane $\mathbb{C} \setminus [1,\infty)$: 
\begin{equation*}
{}_2F_1\left(\frac{1}{2}+\alpha, \alpha, \frac{1}{2}; 1- \frac{\cosh^2(r)}{\cosh^2(x)}\right) = \sum_{j\geq 0} \frac{(\alpha)_j(\alpha+1/2)_j}{(1/2)_j j!} \left(1- \frac{\cosh^2(r)}{\cosh^2(x)}\right)^j, \quad \zeta \leq x.
\end{equation*}
Consequently, the generalized binomial Theorem yields following expansion:
\begin{multline}\label{eq-2F1}
\left(\frac{\cosh^{2}(r)}{\cosh^{2}(x)}\right)^{\alpha} {}_2F_1\left(\frac{1}{2}+\alpha, \alpha, \frac{1}{2}; 1- \frac{\cosh^2(r)}{\cosh^2(x)}\right)  = 
\left(1-\left(1-\frac{\cosh^{2}(r)}{\cosh^{2}(x)}\right)\right)^{\alpha} \\ {}_2F_1\left(\frac{1}{2}+\alpha, \alpha, \frac{1}{2}; 1- \frac{\cosh^2(r)}{\cosh^2(x)}\right) \nonumber
  = 1+ \sum_{j\geq 1} c_j(\alpha) \left(1- \frac{\cosh^2(\zeta)}{\cosh^2(x)}\right)^j, 
\end{multline}
where 
\begin{equation*}
c_j(\alpha) = \frac{1}{j!}\sum_{m=0}^j \binom{j}{m} \frac{(\alpha+1/2)_m(\alpha)_m(-\alpha)_{j-m}}{(1/2)_m}, \quad j \geq 1. 
\end{equation*}
Now, assume $\alpha$ is small enough. Then for any $j-m \geq 1$, one has: 
\begin{equation*}
(-\alpha)_{j-m} = (-\alpha)(1-\alpha)\cdots (j-m-1-\alpha) < 0
\end{equation*}
so that 
\begin{equation*}
c_j(\alpha) < \frac{(\alpha+1/2)_j(\alpha)_j}{j!(1/2)_j}, \quad j \geq 1. 
\end{equation*}
Moreover, the fact that $m \mapsto (\alpha + 1/2)_m/(1/2)_m$ is increasing together with the previous observation show that 
\begin{align*}
c_j(\alpha) \geq \frac{(\alpha + 1/2)_j}{(1/2)_j}\left\{(\alpha)_j + \sum_{m=0}^{j-1}\binom{j}{m}(\alpha)_m(-\alpha)_{j-m}\right\} = 0,
\end{align*}
where we use the fact that the sequence of Pochhammer symbols (rising factorials) is of binomial-type. Using the fact that $s_{t,2n+1}(\cosh(x))$ is a probability density with respect to the volume element 
\begin{equation*}
\frac{2\pi^{n+1/2}}{\Gamma(n+1/2)} \sinh^{2n}(x)
\end{equation*} 
we readily get
\begin{equation*}
\frac{4\pi^{n}}{\Gamma(n)}\int_0^{\infty} (\sinh(r))^{2n-1}\cosh(r) dr \int_{r}^{\infty} \frac{dx \sinh(x)}{\sqrt{\cosh^2(x) - \cosh^2(r)}}s_{t,2n+1}(\cosh(x))  = 1.
\end{equation*}

On the other hand, changing the order of integration, we have: 
\[
\mathbb{E}\left\{\exp-\left(2\alpha^2 \int_0^t J(s) ds \right)\right\} = e^{-2\alpha^2t}\left(1+\sum_{j\ge1}c_j(\alpha)I_j\right).
\]
where we set for any $j\ge1$:
\begin{multline*}
I_j:= \frac{4\pi^{n}}{\Gamma(n)}\int_0^{\infty} (\sinh(r))^{2n-1}\cosh(r) dr \\ 
 \int_{r}^{\infty} \frac{dx \sinh(x)}{\sqrt{\cosh^2(x) - \cosh^2(r)}}s_{t,2n+1}(\cosh(x)) \left(\frac{\cosh^2x-\cosh^2r}{\cosh^2x}\right)^j.
\end{multline*}
Using standard variables changes, we derive: 
\[
\int_0^x\frac{ (\sinh(\zeta))^{2n-1}\cosh(r)}{\sqrt{\cosh^2(x) - \cosh^2(r)}}  \left(\frac{\cosh^2x-\cosh^2r}{\cosh^2x}\right)^jdr
=\frac{(\sinh (x))^{2j+2n-1}}{(\cosh(x))^{2j}}\frac{\Gamma(j+\frac12)\Gamma(n)}{2\Gamma(j+n+\frac12)}
\]
whence
\begin{align*}
I_j&=\frac{4\pi^{n}\Gamma(j+\frac12)}{2\Gamma(j+n+\frac12)}\int_0^{\infty} (\sinh(x))^{2n}(\tanh(x))^{2j}  s_{t,2n+1}(\cosh(x)) dx\\
&=\frac{\Gamma(n+\frac12)\Gamma(j+\frac12)}{\sqrt{\pi}\Gamma(n+j+\frac12)}\E\left( (\tanh(d(0, B_t)))^{2j}\right),
\end{align*} 
where $d(\cdot, \cdot)$ is the Riemannian distance and $B_t$ is the Brownian motion on $H^{2n+1}$. Altogether, we get for small $\alpha$:

\begin{equation*}
\mathbb{E}\left\{\exp-\left(2\alpha^2 \int_0^t J(s) ds \right)\right\} =  e^{-2\alpha^2t} + e^{-2\alpha^2t} \mathbb{E}\left(\sum_{j\ge1}\frac{(1/2)_j}{(n+1/2)_j} c_j(\alpha)(\tanh(d(0, B_t))^{2j}\right).
\end{equation*}
But
\begin{align*}
\sum_{j\ge1}\frac{(1/2)_j}{(n+1/2)_j} c_j(\alpha)(\tanh(d(0, B_t))^{2j} & \leq \sum_{j\ge1}\frac{(\alpha+1/2)_j(\alpha)_j}{j!(n+1/2)_j} (\tanh(d(0, B_t))^{2j}) 
\\& \leq \frac{\alpha}{n+1/2} \left(\alpha+\frac{1}{2}\right)  {}_2F_1\left(\alpha+1, \alpha + \frac{3}{2}, n+\frac{3}{2}; \tanh(d(0, B_t))^{2}\right)
\\& \leq \frac{\alpha(2\alpha+1)}{2n+1} {}_2F_1\left(\alpha+1, \alpha + \frac{3}{2}, n+\frac{3}{2}; 1\right)
\\& = \frac{\alpha(2\alpha+1)}{2n+1}\frac{\Gamma(n+3/2)\Gamma(n-1-\alpha)}{\Gamma(n-\alpha)\Gamma(n-\alpha+1/2)}
\end{align*}
where the last line follows from Gauss hypergeometric Theorem (\cite{Erd}). Consequently, the dominated convergence theorem entails 
\begin{align*}
\lim_{\alpha \rightarrow 0^+} \mathbb{E}\left(\sum_{j\ge1}\frac{(1/2)_j}{(n+1/2)_j} c_j(\alpha)(\tanh(d(0, B_t))^{2j}\right) = 0 
\end{align*}
which in turn allows to recover our limit theorem in Proposition \ref{Lemma1}.

\section{Appendix B}
Here, we provide some details of the computations leading to the identities. 
\begin{align*}
d \mathrm{tr} (\eta)  & = \mathrm{tr}\left[(I_k -w^*w)^{-1}(dw)^*(I_{n-k} - ww^*)^{-1} (dw)\right] 
\\& =-\partial \overline{\partial} \ln \det (I_k -w^*w). 
\end{align*} 
We shall use the analogue of Lemma 1 in \cite{BZ} for complex Hermitian invertible matrices which asserts that: 
\begin{equation*}
\partial_{M_{qj}} M^{-1}_{pl}  = -M^{-1}_{pq} M^{-1}_{jl}, \quad \partial_{M_{qj}} \log[\det(M)] = M^{-1}_{jq}.
\end{equation*}
Now, recall that 
\begin{equation*}
-\mathrm{tr} (\eta) = \frac12 \mathrm{tr} \left[(I_k-w^*w)^{-1} ( dw^* \, w-w^* dw)\right], 
\end{equation*}
and take $M := (I_k-w^*w)$. Then the chain rule yields the following identities: 
\begin{eqnarray*}
\partial M^{-1} (dw)^* w & = & M^{-1} w^* (dw) M^{-1} (dw)^* w \\ 
\overline{\partial} M^{-1} (dw^*) w & = & M^{-1} (dw)^* w M^{-1} (dw)^* w \\ 
\partial M^{-1} w^* (dw)  & = & M^{-1} w^* (dw) M^{-1} w^* (dw) \\  
\overline{\partial} M^{-1} w^* (dw)  & = & M^{-1} (dw)^* w M^{-1} w^* (dw).
\end{eqnarray*}   
Next, we use the fact that the exterior product of one-forms is alternating to see that: 
\begin{align*}
\mathrm{tr} \left[\overline{\partial} M^{-1} (dw)^* w\right] = \mathrm{tr} \left[\partial M^{-1} w^* (dw)\right] = 0. 
\end{align*}
As a result, we get: 
\begin{align*}
-2d\mathrm{tr} (\eta) & =  \mathrm{tr}\left[M^{-1} w^* (dw) M^{-1} (dw)^* w\right] - \mathrm{tr}\left[M^{-1} (dw)^* wM^{-1} w^* (dw)\right] - 2\mathrm{tr}\left[M^{-1} (dw)^* (dw)\right]
\\& = 2\mathrm{tr}\left[M^{-1} w^* (dw) M^{-1} (dw)^* w\right] +  2\mathrm{tr}\left[(dw) M^{-1} (dw)^*\right]
\\& = 2 \mathrm{tr}\left[(I_{n-k} + wM^{-1}w^*) (dw) M^{-1} (dw)^*\right]. 
\end{align*}
Moreover, we note that 
\begin{align*}
(I_{n-k} + wM^{-1}w^*)(I_{n-k} - ww^*) &= I_{n-k} - ww^* + wM^{-1}w^* - wM^{-1}(w^*w)w^* 
\\& = I_{n-k} - ww^* + wM^{-1}w^* - wM^{-1}(w^*w - I_k + I_k)w^* = I_{n-k}, 
\end{align*}
so that $I_{n-k} + wM^{-1}w^* = (I_{n-k} - ww^*)^{-1}$, and in turn: 
\begin{align*}
-\mathrm{tr} (\eta) & = \mathrm{tr}\left[(I_{n-k} - ww^*)^{-1}(dw) (I_k-w^*w)^{-1} (dw)^*\right]
\\& = - \mathrm{tr}\left[(I_k-w^*w)^{-1} (dw)^*(I_{n-k} - ww^*)^{-1}(dw) \right],
\end{align*}
as claimed.

We proceed now to the proof of 
\begin{equation*}
d \mathrm{tr} (\eta) = -\partial \overline{\partial} \ln \det (I_k -w^*w) = -\partial \overline{\partial} \ln \det(M).
\end{equation*}
To this end, we derive: 
\begin{equation*}
\partial_{qm}\ln \det(M) = \sum_{l} (w^*)_{lq}M^{-1}_{ml}dw_{qm} 
\end{equation*}
whence we deduce 
\begin{align*}
\overline{\partial_{rj}}\partial_{qm}\ln \det(M) = -\delta_{rq}M^{-1}_{mj} d\overline{w}_{rj} \wedge dw_{qm} - \sum_{l,s}(w^*)_{lq}M^{-1}_{mj}M^{-1}_{sl}d\overline{w}_{rj} \wedge dw_{qm}.
\end{align*}
Summing over $r,j,q,m$, we get 
\begin{align*}
\partial \overline{\partial} \ln \det(M) = \mathrm{tr}\left[(dw)M^{-1}(dw)^*\right] + \mathrm{tr}\left[M^{-1} w^* (dw) M^{-1} (dw)^* w\right] = - d\mathrm{tr} (\eta).
\end{align*}

\bibliographystyle{amsplain}
\bibliography{biblio}

\end{document}